\documentclass[11pt]{article}

\usepackage{amsmath, amsthm, amssymb, tikz, hyperref, enumerate}
\usepackage{amsfonts}
\usepackage{fullpage}
\usepackage[enableskew,vcentermath]{youngtab}

\newcommand\beq{\begin{equation}}
\newcommand\eeq{\end{equation}}
\newcommand\bce{\begin{center}}
\newcommand\ece{\end{center}}
\newcommand\bea{\begin{eqnarray}}
\newcommand\eea{\end{eqnarray}}
\newcommand\ba{\begin{array}}
\newcommand\ea{\end{array}}
\newcommand\ben{\begin{enumerate}}
\newcommand\een{\end{enumerate}}
\newcommand\bit{\begin{itemize}}
\newcommand\eit{\end{itemize}}
\newcommand\brr{\begin{array}}
\newcommand\err{\end{array}}
\newcommand\bt{\begin{tabular}}
\newcommand\et{\end{tabular}}

\newtheorem{theorem}{Theorem}[section]
\newtheorem*{theorem*}{Theorem}

\newtheorem{proposition}[theorem]{Proposition}
\newtheorem{lemma}[theorem]{Lemma}

\newtheorem{cor}[theorem]{Corollary}
\newtheorem{definition}[theorem]{Definition}
\newtheorem{example}[theorem]{Example}

\newtheorem{obs}[theorem]{Observation}

\newtheorem{defn}[theorem]{Definition}

\newtheorem{remark}[theorem]{Remark}

\newtheorem{observation}[theorem]{Observation}

\newcommand{\todo}[1]{\vspace{5 mm}\par \noindent
\marginpar{\textsc{ToDo}} \framebox{\begin{minipage}[c]{0.95
\textwidth}
 \end{minipage}}\vspace{5 mm}\par}

\begin{document}
\title{On the poset of king-non-attacking permutations}
\author{Eli Bagno, Estrella Eisenberg, Shulamit Reches and Moriah Sigron}

\maketitle
\begin{abstract}
A king-non-attacking permutation is a permutation $\pi \in S_n$ such that $|\pi(i)-\pi(i-1)|\neq 1$ for each $i \in \{2,\dots,n\}$.
We investigate the structure of the poset of these permutations under the containment relation,
and also provide some results on its M\"obius function.  
\end{abstract}
\section{Introduction}
Hertzsprung's problem \cite{Sl} is to find the number of ways to arrange $n$ non-attacking kings on an $n \times n$ chess board such that each row and each column contains exactly one king.  
Let $S_n$ be the symmetric group on $n$ elements. 
In this paper we write permutations in a one line notation, $\sigma=[\sigma_1,\dots,\sigma_n]\in S_n$, though, in the examples, when possible we omit the commas.  

By identifying a permutation $\sigma=[\sigma_1,\dots,\sigma_n]\in S_n$  with its plot, i.e. the set of all lattice points of the form $(i,\sigma_i)$ where $1 \leq i \leq n$, the problem of placing $n$ non-attacking kings reduces to finding the number of permutations $\sigma \in S_n$ such that for each $1 < i  \leq n$, $|\sigma_i-\sigma_{i-1}|>1$. This set is counted in OEIS A002464.

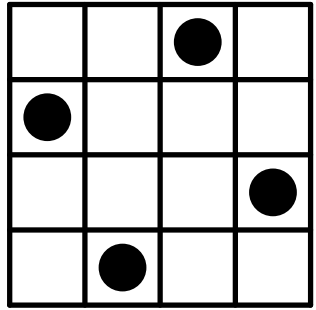
\begin{figure}[!ht]
  \centering
\begin{tikzpicture}[cap=round,line width=2pt]
\draw[-](0,0)--(4,0);
\draw[-](0,1)--(4,1);
\draw[-](0,2)--(4,2);
\draw[-](0,3)--(4,3);
\draw[-](0,4)--(4,4);
\draw[-](0,0)--(0,4);
\draw[-](1,0)--(1,4);
\draw[-](2,0)--(2,4);
\draw[-](3,0)--(3,4);
\draw[-](4,0)--(4,4);
\draw[line width=1pt][fill]  (0.5,2.5) circle (0.3cm);
\draw[line width=1pt][fill]  (1.5,0.5) circle (0.3cm);
\draw[line width=1pt][fill]  (2.5,3.5) circle (0.3cm);
\draw[line width=1pt][fill]  (3.5,1.5) circle (0.3cm);

\end{tikzpicture}
\caption{The plot of $[3142]$.}\label{fig1}
\end{figure}

Let $K_n$ be the set of all such permutations in $S_n$ and let us denote $\mathcal{K}= \cup_{n \in \mathbb{N}}K_n$. In this paper we call them simply {\it king permutations} or just {\it kings}.  
For example: 
$K_1=S_1, K_2=K_3=\emptyset$, $K_4=\{[3142],[2413]\}$. Observe that $K_n$ is closed under the inverse action, as can be seen in the grid, where this action corresponds to the reflection through the line $y=x$ and also under the reverse action: $\sigma_i \mapsto \sigma_{n+1-i}$.

An explicit formula for the number of king permutations was given by Robbins \cite{Ro}. He also showed that when $n$ tends to infinity, 
the probability of picking such a permutation from $S_n$ approaches $e^{-2}$.

In the table below, we present the number of king permutations
of order $n$ for low values of $n$.  
\begin{center}
\begin{tabular}{r||r|r|r|r|r|r|r|r|r|r|r|r}
  $n $ &  1 & 2 & 3 & 4 & 5 & 6 & 7 & 8 & 9 \\
  \hline\hline
 $|K_n|$  & 1 & 0  & 0  & 2  & 14  & 90  & 646  & 5242 & 47622    \\
   
\end{tabular}
\end{center}

Let $\sigma,\pi \in \bigcup_{n \in \mathbb{N}}{S_n}$. We say that $\sigma$ {\it contains} $\pi$ if there is a subsequence of the one line notation of $\sigma$ that is order-isomorphic to that of $\pi$. As an example, $[3624715]$ contains $[3142]$, as both the subsequences $6275$ and $6475$ testify. 
If $\pi$ is contained in $\sigma$, then we write $\pi \preceq \sigma$, while if in 
addition $\pi\neq \sigma$, it will be written as $\pi \prec \sigma$.

The set of all permutations $\cup_{n \in \mathbb{N}} S_n$ is a poset under the partial order given by containment.
It has been studied in numerous works including \cite{BJKM},\cite{BM}.

We are interested in the sub-poset $\mathcal{K}=\cup_{n \in \mathbb{N}}K_n$ containing only the king permutations. Its minimal element is the identity permutation $[1]$, and in the next level appear $[2413]$ and $[3142]$.


In order to analyze properties of the poset we are dealing with, one can use the Manhattan or taxicab distance which is defined as follows:
\begin{definition}
Let $\sigma \in S_n$ and let $i,j\in [n]$. The (Manhattan) distance between the entries $\sigma_i$ and $\sigma_j$ is defined to be the $L_1$ distance between the corresponding points in the plot of $\sigma$:

$$d_{\sigma}(i,j)=||(i,\sigma_i)-(j,\sigma_j)||_1=|i-j|+|\sigma_i-\sigma_j|.$$
\end{definition}

The {\em breadth} of an element $\sigma \in S_n$ is defined in 
\cite{BHE} to be:
$$br(\sigma)=min_{i,j\in [n],i\neq j}d_{\sigma}(i,j).$$

It is easy to observe that for $n>1$ we have $\pi \in K_n$ if and only if $br(\pi)\geq 3$.
In a paper by Bevan, Homberger and  Tenner, \cite{BHE}, the authors define the notion of a {\it $k$-prolific permutation}. A permutation $\pi \in S_n$ is called $k$-prolific if each $(n-k)$-subset of the letters of $\pi=[\pi_1,\dots,\pi_n]$ forms a unique pattern.
It is shown there that $\sigma$ is $k$-prolific if and only if $br(\sigma)\geq k+2$. Hence, $\sigma \in K_n$ if and only if it is $1$- prolific. 
Indeed, to fail to be $1$-prolific means that there are $(n-1)$-subsequences of $\sigma$ that are order isomorphic, meaning that two adjacent positions in $\sigma$ are occupied by adjacent values.

\subsection{Overview and main results}

The first main result in this paper claims that the permutations $[2413]$ and $[3142]$ serve as building blocks of the poset of king permutations. 
Explicitly, we claim that for each $\pi \in K_n$ ($n \geq 4$), 
either $[2413] \preceq \pi$ or $[3142] \preceq \pi$.
Next, we have some results that focus on the structure of the poset of king permutations. Here, the first result deals with the density of the poset, claiming that for each pair of king permutations $\pi \prec \sigma$, there exists a chain of king permutations $\pi=\pi^0 \prec \pi^1 \cdots \prec \pi^k=\sigma$ such that $|\pi^i|-|\pi^{i-1}| \in \{1,3\}$,
 Where $|\pi|=l$ if $\pi \in K_l$. 
 
 We observe that the chains in the poset $\mathcal{K}$ might be saturated.  In order to characterize the permutations of which this phenomenon occurs we define the concept of a prince of a permutation. Namely, $\pi \in K_{n-1}$ is called a {\it prince} of $\sigma \in K_n$ if $\pi \prec \sigma$. 
We give a complete description of all the permutations which have no princes. We show that the only possibility for a king to have no princes is where it can be presented as an inflation of kings of the forms $[2413]$ and $[3142]$. 
Then, we prove that $K_5$ intersects the downset of each king. 
Finally, dealing with the 
M\"obius function, $\mu$, of the poset of king permutations, we prove that the $\mu$ vanishes for each king permutation which does not contain both $[2413]$ and $[3142]$.

The rest of the paper is organized as follows.
Section \ref{background} contains background material including blocks,
simple permutations,
inflation, and the M\"obius function.
In Section \ref{structure} we present our main results regarding the structure of the poset of the king permutations. 
In Section~\ref{mobious} we consider the M\"obius function of the poset of the king permutations and we introduce some results regarding this poset.

\section{Background}\label{background}
In order to better understand the structure of the poset of king permutations, we start by 
presenting some preliminaries concerning simple permutations and the M\"obius function.  Original papers will be mentioned occasionally, but terminology and notation will follow (with a few convenient exceptions) the recent survey \cite{Vatter}.

\begin{definition} \label{blocks}
Let $\pi=[\pi_1, \ldots ,\pi_n] \in S_n$.   
A {\em block} (or {\em interval}) of $\pi$ is a nonempty contiguous
sequence of entries $\pi_i \pi_{i+1} \ldots \pi_{i+k}$ whose values also form a contiguous sequence of integers. A block of length $k$ will be also called a {\em $k$-block}, and it will be called a {\em strict $k$-block} if it is not contained in a $k+1$ block. 
\end{definition}

\begin{example}
If $\pi = [2647513]$ then $6475$ is a block but $64751$ is not. 
\end{example}

Each permutation can be decomposed into singleton blocks, and also forms a single block by itself; these are the {\em trivial blocks} of the permutation. All other blocks are called {\em proper}.

\begin{definition}\label{simple permutation}
A permutation 
is {\em simple} if it has no proper blocks. 
\end{definition}

\begin{example}\label{ex:simple up to 5}
The permutation $[3517246]$ is simple. 
\end{example}

Note that a permutation is in $K_n$ iff it has no $2$-block. Therefore, the set of simple permutations \cite{Adin} is a pure subset of the set of king permutations. Clearly, there are king permutations which are not simple, for example $[31425]$. 


\begin{definition}
A {\em block decomposition} of a permutation is a partition of it into disjoint blocks. 
\end{definition}

For example, the permutation $\sigma=[67183524]$  can be decomposed as $67\ 1\ 8\ 3524$. 
In this example, the relative order between the blocks forms the permutation $[3142]$, i.e., if we take for each block one of its elements as a representative then the sequence of representatives is order-isomorphic to $[3142]$. 
Moreover, the block $67$ is order-isomorphic to $[12]$, and the block $[3524]$ is order-isomorphic to $[2413]$. These are instances of the concept of {\em inflation}, defined as follows.

\begin{definition}
\label{inflation}
Let $n_1, \ldots, n_k$ be positive integers with $n_1 + \cdots + n_k = n$.
The {\em inflation} of a permutation $\pi \in S_k$ by permutations $\alpha^i \in S_{n_i}$ $(1 \leq i \leq k)$ is 
the permutation $\pi[\alpha^1, \ldots, \alpha^k] \in S_n$ obtained by replacing the $i$-th entry of $\pi$ by a block which is order-isomorphic to the permutation $\alpha^i$
on the numbers $\{s_i + 1, \ldots, s_i + n_i\}$ instead of $\{1, \ldots, n_i\}$, where $s_i = n_1 + \cdots + n_{i-1}$ $(1 \leq i \leq k)$. 
\end{definition}

\begin{example}
The inflation of $[2413]$ by $[213],[21],[132]$ and $[1]$ is 
\[
2413[213,21,132,1]=546  98  132  7\in S_9.
\]
\end{example}

\begin{definition}
 Let $\sigma \in K_n$. An element $\tau \in K_{n-1}$ will be called a prince of $\sigma$
 if $\tau \prec \sigma$. In other words, if $\tau$ is obtained  by omitting exactly one element 
 from $\sigma$ and is still a king permutation then we call $\tau$ a prince of $\sigma$. In this case we also
 say that $\sigma$ has a prince. 
 \end{definition}
 
 \begin{example}
The king permutation $\tau=[41352]$ is a prince of the king permutation $\sigma=[524613]$. 
$\pi=[3142]$ is prince of $\tau$, but  not a prince of $\sigma$.
\end{example}

In order to define the M\"obius function of the poset $\mathcal{K}$, we recall the definition of an interval. 

\begin{definition}
 The closed interval $[\tau,\sigma]$ is defined as:
$$[\tau,\sigma]=\{\pi \in \mathcal{K} \mid \tau \preceq \pi \preceq \sigma\}.$$
The half-open interval is defined as:
$$[\tau,\sigma)=\{\pi \in \mathcal{K} \mid \tau \preceq \pi \prec \sigma\}.$$
\end{definition}

Now, we can define the M\"obius function for the poset $\mathcal{K}$.
\begin{definition}

\[
\mu(\tau,\sigma) =
\begin{cases} 
0, &  \tau \npreceq  \sigma; \\
1, &  \tau =\sigma; \\
-\sum\limits_{\pi \in [\tau,\sigma) }\mu(\tau,\pi), & \text{Otherwise}. 
\end{cases}
\]
When $\tau=[1]$, the identity permutation of length $1$, we write $\mu(\pi):=\mu([1],\pi)$. 

\end{definition}





\section{The structure of the poset of king permutations}
\label{structure}

In this section we study the structure of the poset of king permutations with respect to the containment relation. We start with the building blocks. 

The elements $[2413]$ and $[3142]$ in the poset of simple permutations (see Definition \ref{simple permutation}) have a special role as every simple permutation must contain at least one of them (see \cite{AD}). As we show below, the same is true for king permutations. In order to do this, we start with the following definitions. 

\begin{definition}
For each permutation $\pi \in S_n,$ $ \pi=[\pi_1,\pi_2,\ldots,\pi_n]$ and $i \in \{1,2,\ldots,n\}$, define the permutation $\nabla_i(\pi) \in S_{n-1}$ by deleting the $i-$th entry of $\pi$ and standardizing the remaining entries. More precisely:
 for each $1 \leq k \leq n-1$, the $k-$th entry of $\nabla_i(\pi)$ is

for $k < i$:

$$\big(\nabla_i(\pi)\big)_k=
\begin{cases}
\pi_k, &  \pi_k<\pi_i \\
 \pi_k -1 , & \pi_k>\pi_i
\end{cases} \\$$

and for $k \geq i$:

$$\big(\nabla_i(\pi)\big)_k=
\begin{cases}
\pi_{k+1}, &  \pi_{k+1}<\pi_i \\
 \pi_{k+1} -1 , & \pi_{k+1} >\pi_i
\end{cases} \\$$

Note that if we omit from $\pi$ the element $i$ rather then the $i$-th entry, we actually omit the entry $(\pi^{-1})_i$ from $\pi$, and we notate  $\nabla_{\pi^{-1}(i)}(\pi)$ by  $\nabla^*_i(\pi)$.
 
For example, if $\pi=[641325]$ then to omit the element $4$, we use $\nabla^*_4(\pi)=\nabla_2(\pi)=[51324]$.   
\end{definition}


If one omits two elements from a permutation $\pi$ in a sequence, 
it is more convenient notation-wise to omit first the rightmost of the two (in the case of $\nabla$) or the larger (in the case of $\nabla^*)$. This is the content of the next observation:

\begin{observation}\label{omit the bigger first}
\begin{enumerate}
Let $\pi \in S_n$. Then for each $1 \leq j<i \leq n$:

\item $\nabla_j(\nabla_i(\pi))=\nabla_{i-1}(\nabla_j(\pi))$. 
\item $\nabla^*_j(\nabla^*_{i}(\pi))=\nabla^*_{i-1}(\nabla^*_{j}(\pi))$
\end{enumerate}

\end{observation}

\begin{example}

\end{example}
\begin{tikzpicture}
        \tikzstyle{every node} = [rectangle]
        \node (7426153) at (0,0) {$[7 {\bf4} 26 {\bf1} 53]$};
        \node at (-2.1,-1) {$\nabla^*_4=\nabla_2$};
        \node (625143) at (-2.2,-2) {$[625 {\bf1} 43]$};
        \node at (1.95,-1) {$\nabla^*_1=\nabla_5$};
        \node (631542) at (2.2,-2) {$[6 {\bf3} 1542]$};
         \node at (-2.1,-3) {$\nabla^*_1=\nabla_4$};
         \node at (2.2,-3) {$\nabla^*_3=\nabla_2$};
        \node  (51432) at (0,-4) {$[51432]$}; 
        
        \foreach \from/\to in {7426153/625143, 7426153/631542,625143/51432,631542/51432}
            \draw[->] (\from) -- (\to);
    \end{tikzpicture}

It is easy to see that omitting either element of a $2$-block results in the same permutation. 
\begin{obs}\label{2-block}
If $\sigma$ has a $2$-block $\sigma_j=a$, $\sigma_{j+1}=a \pm 1$, then $\nabla_j(\sigma)=\nabla_{j+1}(\sigma)$.
\end{obs}

In order to investigate the structure of the poset of king permutations we introduce a new concept called a 
{\it separator}. A comprehensive discussion of this concept can be found in \cite{BERS}.

\begin{definition}\label{def separate}
For $\sigma=[\sigma_1,\ldots,\sigma_n] \in S_n$ we say that $\sigma_i={\bf \textcolor{red}{a}}$ {\em separates} $\sigma_{j_1}$ from $\sigma_{j_2}$ in $\sigma$ if by omitting $a$ from $\sigma$ we get a {\bf new} $2-$block. This happens if and only if one of the following cases holds: 

\begin{enumerate}
\item $j_1,i,j_2$ are subsequent numbers and $|\sigma_{j_1}-\sigma_{j_2}|=1$, i.e $$\sigma=[\ldots,{\bf b},{\bf \textcolor{red}{a}},{\bf b \pm 1},\ldots]$$
In this case we say that $a$ is a {\bf vertical separator}.
\item $\sigma_{j_1},\sigma_i=a,\sigma_{j_2}$ are subsequent numbers and $|j_1-j_2|=1$, i.e, $$\sigma=[\ldots,{\bf \textcolor{red}{a}},\ldots,{\bf a \pm1,a\mp1},\ldots]$$ or $$\sigma=[\ldots,{\bf a\pm1,a\mp1},\ldots,{\bf \textcolor{red}{a}},\ldots].$$
In this case we say that $a$ is a {\bf horizontal separator}.
\end{enumerate}

\end{definition}

Note that in the term \textbf{separator} we refer to the \textbf{value} of the element and \textbf{not} to its position.

\begin{defn}
Let $Sep_v(\pi)$ and $Sep_h(\pi)$ be the sets of vertical and horizontal separators of $\pi$,  respectively. 
\end{defn}

\begin{example}
Let $\sigma=[132465879]$.   Then $Sep_{v}(\sigma)=\{3,2,6,7\}$, and $Sep_h(\sigma)=\{3,2,5,8\}$. 
Note that $7$ is a vertical separator, even though $7$ is a part of a $2-$block: $87$, since by omitting $7$ from $\sigma$ we get a $\bf{new}$ $2-$block: $78$. 
\end{example}

\begin{remark}\label{obs on Sep}
Several comments are now in order:
\begin{enumerate}

\item Notice the significance of the word {\bf new} in Definition \ref{def separate}. For example, the identity permutation has plenty of $2$-blocks even though it has no separators.
\item The numbers $1$ and $n$ can only be vertical separators, while  $\sigma_1$ and $\sigma_n$ can only be horizontal separators.
\item If $\sigma_i$ is a vertical separator in $\sigma$ then $i$ is a horizontal separator in $\sigma^{-1}$. Hence $Sep_v(\sigma)=Sep_h(\sigma^{-1})$

\item $Sep_v(\sigma)=Sep_v(\sigma^r)$ and  $Sep_h(\sigma)=Sep_h(\sigma^r)$ where $\sigma^r$ is the reverse of $\sigma$.  

\end{enumerate} 
\end{remark}

In the following theorem we present our first main result which claims that the permutations $[2413]$ and $[3142]$ serve as building blocks for the poset of king permutations. 

\begin{theorem}\label{Each king contains 3142 or 2413}
For every $\pi \in K_n$ ($n \geq 4$), either $[2413] \preceq \pi$ or $[3142] \preceq \pi$. 
\end{theorem}

\begin{proof}
By induction on $n \geq 4$, we will prove that for every $n$, and for every $\pi \in K_n$, either $[2413] \preceq \pi$ or $[3142] \preceq \pi$.  For $n=4$ this is trivial since $K_4=\{[3142],[2413]\}$. Now assume $n > 4$, and assume that each king permutation of order $n-1$ contains at least one of the permutations $[2413],[3142]$. We will prove that each king permutation of order $n$ contains at least one of these patterns, too. Let $n > 4$ and assume to the contrary that there is $\pi \in K_n$ which does not contain either $[2413]$ or $[3142]$. Then $\sigma=\nabla_1^*(\pi)$ contains neither of them as well. By the induction hypotheses, $\sigma \notin K_{n-1}$, which implies, by Remark \ref{obs on Sep}, that $1$ is a vertical separator.
Hence we must have: $\pi=[\ldots, a,1,a+1,\ldots]$ or $\pi=[\ldots, a+1,1,a,\ldots]$. 
Without loss of generality, assume that  $\pi=[\ldots, a,1,a+1,\ldots]$. 

Let us check the location of $2$ in $\pi$. Obviously, $a=2$ or $a+1=2$ is impossible since $\pi \in K_n$, so we may assume that $a>2$. If $2$ is located to the right of $a+1$, then $\pi=[\ldots ,a,1,a+1, \ldots ,2 ,\ldots]$ and so $a,1,a+1,2$ is a $[3142]$– pattern, so we may assume that $2$ is located to the left of $a$, thus $\pi=[\ldots, 2, \ldots, a,1,a+1, \ldots]$. Consider now $\tau=\nabla_{a+1}^*(\pi)$, the result of removing $a+1$ from $\pi$. Again, according to the induction hypothesis, $\tau$ is not a king permutation and thus, must contain a $2$-block. Clearly, $a+1$
 is not a vertical separator since the element $2$ is far left. Therefore, $a+1$ is a horizontal separator. Hence $\pi=[\ldots, 2, \ldots, a+2,a,1,a+1, \ldots]$ and $2,a+2,1,a+1$ forms a $[2413]$ pattern, which contradicts our assumption. 


\end{proof}

\begin{remark}
In \cite{BBL}, Bose, Buss, and Lubiw showed that every permutation that avoids $[3142]$ and $[2413]$ is {\emph separable}. Therefore, Theorem \ref{Each king contains 3142 or 2413} is equivalent to saying that no king permutation is separable. 
\end{remark}

The following lemma refers to the separators, and will be used later in some theorems related to the structure of the poset of the King permutations.

\begin{lemma}\label{i and j}
Let $\sigma\in K_n$ and let $1 \leq j< i\leq n$. If $\nabla_j(\nabla_i(\sigma))\in K_{n-2}$ but $\nabla_i(\sigma)\notin K_{n-1}$ and $\nabla_j(\sigma)\notin K_{n-1}$ then $\sigma_i$ separates $\sigma_j$  
from some element and $\sigma_j$ separates $\sigma_i$ from some element. 
\end{lemma}
\begin{proof}
If $\nabla_i(\sigma)\notin K_{n-1}$ then $\sigma_i$ is a separator, i.e., there is a block in $\nabla_i(\sigma)$. Now, $\nabla_j(\nabla_i(\sigma))\in K_{n-2}$, which can happen only if $\sigma_j$ is a part of this block.

Recall that by Observation \ref{omit the bigger first}, $\nabla_j(\nabla_i(\sigma))=\nabla_
{i-1}(\nabla_j(\sigma))$ and the proof is complete by interchanging the rules of $\sigma_i$ and $\sigma_j$.   

\end{proof}

\begin{lemma}\label{same type}
Let $\sigma \in K_n$. Assume that $\sigma_i$ separates $\sigma_j$ from some element and that $\sigma_j$ separates $\sigma_i$ from some element. Then $\sigma_i$ and $\sigma_j$ are separators of the same type. 
\end{lemma}
\begin{proof}


Assume to the contrary that $\sigma_i$ and $\sigma_j$ are separators of
different types. Without loss of generality, $\sigma_i$ is a vertical separator. In this case we must have $j=i \pm 1$ and $|\sigma_{i+1}-\sigma_{i-1}|=1$. The case $j=i+1$ is 

$\sigma = \left[  \begin{array}{cccccclclcc}
1 & \cdots & i-1 & i & j=i+1 & \cdots\\
\sigma_1 & \cdots & a & b & a\pm 1 & \cdots \\
\end{array}\right]. $\\

Now, if $\sigma_j$ is a horizontal separator, then it separates $\sigma_j+1$ from $\sigma_j-1$. Therefore $\sigma_i=\sigma_j \pm 1$ which  contradicts the fact that $\sigma \in K_n$. 
\end{proof}

We observe now that if we omit a vertical separator, (separating $a$ from $a+1$) from a permutation $\sigma \in S_n$ and then we omit either of the elements of the resulted new block, or we first omit either of the set $\{a,a+1\}$ and then we remove the separator, the result is the same permutation. 
The following picture depicts this situation where the separator is bolded, the numbers that it separates are in red, and the elements written next to the edges are the ones we remove:

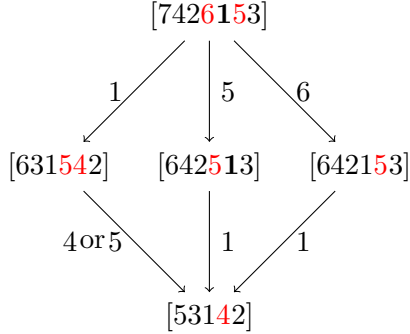
\begin{tikzpicture}
        \tikzstyle{every node} = [rectangle]
        \node (7426153) at (0,0) {$[742{\textcolor{red}6}{\bf 1}{\textcolor{red}5}3]$};
        \node at (-1.25,-1) {1};
        \node (631542) at (-2,-2) {$[631{\textcolor{red}{54}}2]$};
        \node at (0.25,-1) {5};
        \node (642513) at (0,-2) {$[642\textcolor{red}{5}{\bf 1}3]$};
        \node (642153) at (2,-2) {$[6421{\textcolor{red}{5}}3]$};
         \node at (-1.85,-3) {4};
         \node at (-1.55,-3) {or};
         \node at (-1.25,-3) {5};
        \node  (53142) at (0,-4) {$[531\textcolor{red}{4}2]$}; 
        \node at (0.25,-3) {1};
        \node at (1.25,-3) {1};
        \node at (1.25,-1) {6};
        
        \foreach \from/\to in {7426153/631542, 7426153/642513,642513/53142,631542/53142,7426153/642153,642153/53142}
            \draw[->] (\from) -- (\to);
    \end{tikzpicture}

The horizontal separators produce a similar phenomenon which is depicted in the following picture:

 
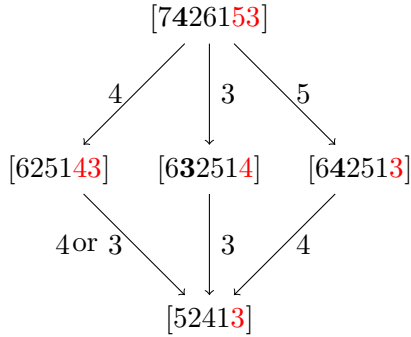
\begin{tikzpicture}
        \tikzstyle{every node} = [rectangle]
        \node (7426153) at (0,0) {$[7{\bf 4}261\textcolor{red}{53}]$};
        \node at (-1.25,-1) {4};
        \node (625143) at (-2,-2) {$[6251\textcolor{red} {43}]$};
        \node at (-1.25,-3) {3};
        \node at (-1.65,-3) {or};
        \node at (-1.95,-3) {4};
        \node (52413) at (0,-4) {$[5241\textcolor{red}3]$};
        \node at (1.25,-3) {4};
        \node at (1.25,-1) {5};
        \node (642513) at (2,-2) {$[6{\bf4} 251 {\textcolor{red}3}]$};
        \node at (0.25,-1) {3};
        \node (632514) at (0,-2) {$[6{\bf3} 251 {\textcolor{red}4}]$};
        \node at (0.25,-3) {3};

        \foreach \from/\to in {7426153/625143,625143/52413,7426153/632514,7426153/642513,632514/52413,642513/52413}
            \draw[->] (\from) -- (\to);
    \end{tikzpicture}

 Both claims are expressed formally in the following :

\begin{observation}\label{same block}
 
Let $\sigma \in S_n$, and assume that $\sigma_i$ separates $\sigma_j$ from $\sigma_k$ in $\sigma$ and that  $i>j$. Then:

$$\nabla_j(\nabla_i(\sigma))=\nabla_i(\nabla_k(\sigma))\,\textrm{    if }\,j<i<k$$
and 

$$\nabla_j(\nabla_i(\sigma))=\nabla_k(\nabla_i(\sigma))\textrm{    if }  i>j>k$$
 (By observation \ref{omit the bigger first} it is sufficient to consider only  the case $i>j$.)
\end{observation}

In \cite{AD}, it is proven that the poset of simple permutations is dense in the sense that for any two simple permutations $\sigma \prec \pi$ there exists a chain of simple permutations $\sigma^0=\sigma \prec \sigma^1 \cdots \prec \sigma^k=\pi$ such that $|\sigma^i|-|\sigma^{i-1}| \in \{1,2\}$. 
In our case, we prove in Theorem \ref{1 or 3} that for any two king permutations, such a chain exists with the stipulation that $|\sigma^i|-|\sigma^{i-1}| \in \{1,3\}$. 
The first step in this direction is the following:

\begin{proposition}

\label{grandson implies son}
Let $\sigma \in K_n$ with $n > 4$, and let $\pi \in K_{n-2}$ be such that $\pi \prec \sigma$. Then there exists $\tau \in K_{n-1}$ such that $\pi \prec \tau \prec \sigma$. 
\end{proposition}

\begin{proof}
Let $\sigma=[\sigma_1,\cdots, \sigma_n] \in K_n$, and $\pi \in K_{n-2}$ be such that $\pi \prec \sigma$.

Hence there are some $i>j$ such that $\pi=\nabla_j(\nabla_i(\sigma))$. 

Let $$G=\{(i,j) \mid \nabla_j(\nabla_i(\sigma))=\pi,i>j \}.$$  

For example, let $\sigma=[361425]$ and $\pi=[2413]$. Then $G=\{(4,3),(5,4),(6,5)\}$.

Assume to the contrary that there is no $\tau \in K_{n-1}$ such that $\pi \prec \tau \prec \sigma$. 

According to our assumption, for each $(i,j)\in G$, both $\nabla_i(\sigma)$ 
and $\nabla_j(\sigma)$ are not in $K_{n-1}$, which means that each of them contains a block of order $2$. 
Moreover, by Lemma \ref{i and j}, $\sigma_i$ separates $\sigma_j$ from some element, and similarly $\sigma_j$ separates $\sigma_i$ from some element and thus, by Equation \ref{same type}, both separators are of the same type. 

Let $i_0=\max \{i \mid (i,j)\in G\}$, and assume that $\sigma_{i_0}$ separates $\sigma_p$ from $\sigma_q$. We have two cases:

\begin{enumerate}
\item The separator $\sigma_{i_0}$ is vertical. In this case we have without loss of generality that $p=i_0-1$ and $q=i_0+1$, and $|\sigma_p -\sigma_q|=1$ so that 
$$\sigma = \left[  \begin{array}{cccccclclcc}
1 & \cdots & p & i_0 & q & \cdots\\
\sigma_1 & \cdots & a & b & a \pm 1 & \cdots \\
\end{array}\right], $$\\
and since the omitting of $\sigma_{i_0}$ creates a block, which we must get rid of by removing $\sigma_p$ or $\sigma_q$ in the passage to $\pi \in K_{n-2}$ and $p<i_0<q$, we must have by \ref{same block}, 
$$\nabla_{i_0}(\nabla_q(\sigma))=\nabla_p(\nabla_{i_0}(\sigma))=\pi.$$ 
This means that $(q,i_0) \in G$ which contradicts the maximality of $i_0$. 

\item The separator $\sigma_{i_0}$ is horizontal. Let $\sigma_{i_0}=a+1$, so that 
 w.l.o.g. $\sigma_q=a+2$ and $\sigma_p=a$, i.e., 

$$\sigma = \left[  \begin{array}{cccccclclcc}
1 & \cdots & p & q & \cdots & i_0 & \cdots \\
\sigma_1 & \cdots & a & a+2  & \cdots & a+1 & \cdots \\
\end{array}\right].$$\\

The removal of $\sigma_{i_0}$ creates a block, which we must get rid of by removing $\sigma_p$ or $\sigma_q$ in the passage to $\pi \in K_{n-2}$, thus  $\nabla_p(\nabla_{i_0}(\sigma))\in K_{n-2}$ but $\nabla_{i_0}(\sigma)\notin K_{n-1}$ and $\nabla_p(\pi)\notin K_{n-1}$. As a result, by Lemma \ref{i and j}, we have that $\sigma_{i_0}$ separates $\sigma_p$ from some element and 
 $\sigma_p$ separates $\sigma_{i_0}$ from some element, call it $\sigma_k$. In the same manner, $\sigma_q$ separates 
 $\sigma_{i_0}$ from some $\sigma_m$. 
 
 
 Now, by \ref{same type}, $\sigma_p$ must be a horizontal separator, so that $\sigma_k=a-1$ and $k \in \{i_0-1,i_0+1\}$ and $ \sigma_q$ must be also a horizontal separator so that
$\sigma_m=a+3$ and $m\in \{i_0-1,i_0+1\}$ with $m \neq k$. 

If $m=i_0+1$, then

$$\sigma = \left[  \begin{array}{cccccclclcc}
1 & \cdots & p & q & \cdots & i_0 & i_0+1 & \cdots \\
\sigma_1 & \cdots & a & a+2  & \cdots & a+1 & a+3 & \cdots \\
\end{array}\right]$$\\

 and by \ref{same block}, $\nabla_q(\nabla_{i_0}(\sigma))=\nabla_q(\nabla_{i_0+1}(\sigma))=\pi$. 
 
 Similarly, if $k=i_0+1$, then $\nabla_p(\nabla_{i_0}(\sigma))=\nabla_p(\nabla_{i_0+1}(\sigma))=\pi$.

In any case, we  contradicted the maximality of $i_0$.


 
\end{enumerate}
\end{proof}

\begin{theorem}

\label{1 or 3}

Let $\sigma,\pi$ be king permutations such that $\pi \prec \sigma$ and $|\sigma|-|\pi|>3$. Then there exists a king permutation $\tau$ such that $\pi \prec \tau \prec \sigma$ and $|\sigma| -|\tau| \in \{1,3\}$. 
\end{theorem}

\begin{proof}

Let $\sigma=[\sigma_1,\dots,\sigma_n]$. 
In order to pass from $\sigma$ to $\pi$, we have to remove $d=|\sigma|-|\pi|$ different elements from $\sigma$.  As there is more then one possibility to choose these elements, let $F$ be the set of all sequences of $d$ elements $(l_1,\dots,l_d)$ such that the removal of them  from $\sigma$ achieves $\pi$.  
Note that in this proof we chose to work with the elements of the permutation $\sigma$ rather than with their locations as in previous proofs. This choice is justified by clarity considerations. 
Here is an example of using the sequence $(1,2,7,8)$ for passing from $\sigma=[571386249]$ to $\pi=[31425]$.  
We start by omitting $1$ from $\sigma$ in order to get $\delta_1=\nabla_1^*(\sigma)=[46275138]$. The second step will be to omit the element $2$ from $\sigma$, but since $2>1$, we actually omit the element $1$ (which fills in for $2$) from $\delta_1$ to get
 $\delta_2=[3516427]$. 
 Now, removing $7$ from $\sigma$ amounts to omitting $5$ from $\delta_2$, so we get $\delta_3=[315426]$. 
 The last step is to remove $8$ from $\sigma$, and this is done by omitting $5$ from $\delta_3$ so we have $\delta_4=[31425]=\pi$. 
 
Here is another example:
Again, let $\sigma=[571386249]$ and $\pi=[31425]$, but this time the procedure will be much easier since we use the descending sequence $(8,7,2,1)$. 
We have $\delta_1=[57136248]$, $\delta_2=[5136247],\delta_3=[412536]$ and $\delta_4=[31425]=\pi$. 

\bigskip
We are ready now to get to the proof. 
Let $f=(l_1,\dots,l_d)\in F$ be chosen such that $\ell=l_1$ is maximal among all the elements appearing in the sequences of $F$. 
If $\ell$ is not a separator then by omitting it we get a king permutation $\tau$ such that $|\sigma|-|\tau|=1$ and $\pi \prec \tau \prec \sigma$ and we are done. 
Otherwise, $\ell$ is a separator. We claim that it is a vertical separator but not a horizontal separator. 

Indeed, if $\ell$ is a horizontal separator then the consecutive subsequence $\ell-1,\ell+1$ (or its reverse) appears in $\sigma$. After omitting $\ell$, we get a block which must be removed in our way to $\pi$ since $\pi$ is a king permutation. Hence there is some $f \in F$ which contains $\ell+1$ and this contradicts the maximality of $\ell$.

Back to the proof, since $\ell$ is a vertical separator, there is some element $a$ such that $\sigma$ contains the consecutive sequence $a,\ell,a+1$ or its reverse. By the maximality of $\ell$ and the fact that $\sigma \in K_n$, we get that $\ell>a+2$ . Let $\delta_1=\nabla_\ell^*(\sigma)$ be the permutation obtained from $\sigma$ by omitting the element $\ell$. Without loss of generality, this permutation contains the $2-$block $a,a+1$, and this is \textbf{the only $2$-block} in $\delta_1$. 
Now, let $\delta_2$ be the permutation obtained from $\delta_1$ by omitting the element $a+1$. If $a+1$ is not a separator in $\delta_1$, then by omitting it we get a permutation with no $2-$blocks (i.e., $\delta_2$ is a king permutation), so we have two options: either $\delta_2$ is a king permutation or $a+1$ is a separator in $\delta_1$.

If $\delta_2$ is a king permutation then by Proposition \ref{grandson implies son} there is some $\tau \in K_{n-1}$ such that $\pi \prec \delta_2 \prec \tau \prec \sigma$ and we are done. 
Otherwise, $a+1$ must be a separator in 
$\delta_1$, 
and we have one of the following three options:
\begin{itemize}
    
     \item The element $a+1$ is a vertical and horizontal separator in $\delta_1$.
In this case, we have the following situation:

\begin{tikzpicture}
        \tikzstyle{every node} = [rectangle]
        \node (sigma) at (0,0) {$\sigma=\cdots {\bf a+2,a,\ell,a+1,a-1} \cdots$};
        \node (delta1) at (0,-1) {$\delta_1=\nabla^*_{\ell}(\sigma)=\cdots {\bf a+2,a,a+1,a-1} \cdots$};
       
    \foreach \from/\to in {sigma/delta1}
            \draw[->] (\from) -- (\to);
    \end{tikzpicture}
    
    Now, $a-1$ is not a vertical separator in $\sigma$ due to the location of $a+2$ and is not a horizontal separators since $\ell \neq a-2$. This means that 
    $\nabla^*_{a-1}(\sigma)\in K_{n-1}$ and we are done.

    \item The element $a+1$ is a vertical but not horizontal separator in 
    ${\bf \delta_1}$. In this case, we have the following situation (note that $x \neq a+2$, since otherwise we are back in the former case):
    
    \begin{tikzpicture}
        \tikzstyle{every node} = [rectangle]
        \node (sigma) at (0,0) {$\sigma=\cdots ,{\bf a,\ell,a+1,a-1},\cdots$};
        \node (delta1) at (0,-1) {$\delta_1=\nabla^*_{\ell}(\sigma)=\cdots ,{\bf a,a+1,a-1 }\cdots$};
       
    \foreach \from/\to in {sigma/delta1}
            \draw[->] (\from) -- (\to);
    \end{tikzpicture}
    
      Now, it is clear that $a+1$ is not a separator  in $\sigma$. This means that 
    $\nabla^*_{a+1}(\sigma)\in K_{n-1}$ and we are done.

     \item The element $a+1$ is a horizontal but not vertical separator in 
    ${\bf \delta_1}$. 
    
    In this case, we have the following situation:
    
    \begin{tikzpicture}
        \tikzstyle{every node} = [rectangle]
        \node (sigma) at (0,0) {$\sigma=\cdots {\bf a+2,a,\ell,a+1}\cdots$};
        \node (delta1) at (0,-1) {$\delta_1=\nabla^*_{\ell}(\sigma)=\cdots {\bf a+2,a,a+1 }\cdots$};
    \foreach \from/\to in {sigma/delta1}
            \draw[->] (\from) -- (\to);
    \end{tikzpicture}
    
    Now, if $a$ is not a separator in $\sigma$, then $\nabla^*_a(\sigma) \in K_{n-1}$ and we  are done.
    Otherwise, we have two cases:
    \begin{enumerate}
  
        \item The element $a$ is a vertical (but not horizontal) separator in $\sigma$. 
         In this case we have $$\sigma=\cdots x,{\bf a+2,a,\ell=a+3,a+1},y \cdots.$$
        
        If $\ell = a+3$ gets removed to form $\pi$ (and we know that it does), then two of 
        $\{a,a+1,a+2\}$ must also be removed to form $\pi$, because $\pi$ is a king permutation. If $x$ or $y$ is $a+4$, then that would also have to be removed, for the same reason, but this contradicts the maximality of $\ell$. Thus $x,y \neq a+4$. If $x,y = a-1$, then that would also have to be removed, for the same reason, so $\pi \prec \nabla^*_{a-1}(\sigma) \prec \sigma$ and $\nabla^*_{a-1}(\sigma) \in K_{n-1}$. Otherwise, $\nabla^*_{a+1}\nabla^*_{a+2}\nabla^*_{a+3}(\sigma) \in K_{n-3}$

         
       
       \item The element $a$ is a horizontal separator (with or without being also a vertical separator) in $\sigma$. In this case we have:
       $\sigma=\cdots{\bf  a+2,a,\ell,a+1,a-1}\cdots$.
        
        (Note that if $\ell=a+3$ then $a$ is also a vertical separator). 
      Now,  it is clear that $\pi \prec \nabla^*_{a-1}(\sigma) \prec \sigma$ and $a-1$ is not a separator in $\sigma$. This means that 
  $\nabla^*_{a-1}(\sigma)\in K_{n-1}$ and we are done.
  \end{enumerate}

\end{itemize}
\end{proof}

As a corollary, we get one of our main results:
\begin{cor}\label{chain}
For each two king  permutations $\pi \prec \sigma$ there exists a chain of king permutations $\pi=\pi^0 \prec \pi^1 \cdots \prec \pi^k=\sigma$ such that $|\pi^i|-|\pi^{i-1}| \in \{1,3\}$. 
\end{cor}

 \subsection{Kings without princes}
 Recall that a {\it prince} of a permutation $\sigma \in K_n$ is a king permutation $\tau \in K_{n-1}$ such that $\tau \prec \sigma$. 
 
 We can formulate now the following:
 \begin{lemma}\label{strict 2 block implies a prince}
 If for $\sigma \in K_n$ there is $\pi \in S_{n-1}$ such that $\pi \prec \sigma$ and $\pi$ contains a single strict $2-$block then $\sigma$ has a prince. 
 \end{lemma}
 \begin{proof}
 Let $i \in \{1,\ldots ,n\}$ be such that $\pi=\nabla^*_i(\sigma)$ contains a 
 single strict 2-block. Then, without loss of generality, $\sigma$ contains the subsequence $k,k+1$, where $k-1$ and $k+2$ do not appear adjacent to $k$ and $k+1$, so we can divide $\sigma$ into 4 parts such that $\sigma=[\delta \quad  k \quad k+1 \quad \epsilon]$ where both $\delta$ and $\epsilon$ are sequences that do not contain any $2-$block (and might be empty). If we now remove the element $k+1$ from $\pi$, and standardize, we get $\tau=[\delta' \quad k \quad \epsilon']\in S_{n-2}$ where $k-1$ and $k+1$ do not appear adjacent to $k$ and $\delta',\epsilon'$ do not contain any $2-$block.  This means that $\tau=\nabla^*_{k+1} \nabla^*_i(\sigma)\in K_{n-2}$ and thus by Proposition \ref{grandson implies son}, $\sigma$ has a prince. 
 \end{proof}

 \begin{example}
 Let $\sigma=[361472958]\in K_9$. After omitting the element $5$ and normalizing we get $\nabla^*_5(\sigma)=\pi=[351462{\bf87}]$ which has the strict $2-$block $87$. By removing from $\pi$ the element $8$ (or the element $7$) and standardizing, we get $\nabla^*_8(\pi)=\tau=[3514627]\in K_7$.  
 This implies the existence of a prince of $\sigma$ (for example $[51362847]$).  
The situation is depicted in the following:

\begin{tikzpicture}
        \tikzstyle{every node} = [rectangle]
        \node (sigma) at (0,0) {$\sigma=[361472958]\in K_9$};
        \node (sigmapi) at (0.6,-1) {$\nabla^*_5(\sigma)$};
        \node (nabla5) at (0,-2) {$\pi=[351462{\bf87}] \in S_8-K_8$};
        \node (pitau) at (0.6,-3) {$\nabla^*_8(\pi)$};
        \node (nabla8) at (0,-4) {$\tau=[3514627]\in K_7$};
        
        \node (griralogit) at (4,-2) {$\Longrightarrow$};
        \node (sigma1) at (8,0) {$\sigma=[361472958]\in K_9$};
        \node (sigmareg) at (8.6,-1) {$\nabla^*_3(\sigma)$};
        \node (sigmaprince) at (8,-2) {$[51362847] \in K_8$};

    \foreach \from/\to in {sigma/nabla5, nabla5/nabla8, sigma1/sigmaprince}
            \draw[->] (\from) -- (\to);
  \end{tikzpicture}

 \end{example}

 The permutations $\pi \in K_n$ which have no princes have a special structure. An example for such an element is $\pi=[7,5,8,6,2,4,1,3,10,12,9,11]\in K_{12}$. Note that $\pi$ has
 the following structure: $\pi=213[3142,2413,2413]$. (Recall the definition of inflation from \ref{inflation}).
 Moreover, by removing any value from $\pi$ we get a $3-$block. (by way of illustration, if we remove the element $8$ and standardise we get $[{\bf 7},{\bf 5},{\bf 6},2,4,1,3,9,11,8,10]$). 
 
 The following theorem shows that this is not coincidental. 
 
\begin{theorem}\label{3 equivalents}
The following conditions are equivalent for each $\pi \in K_n$ with $n\geq 4$.
\begin{enumerate}
    \item There are $\pi^1,\dots, \pi^k \in \{[3142],[2413]\}$ and $\pi' \in S_k$ such that $\pi=\pi'[\pi^1,\dots,\pi^k]$.
    \item For each $i \in \{1,\dots,n\}$, by removing $i$ from $\pi$, we get a block of length $3$. 
    \item $\pi$ has no princes.
\end{enumerate}

\end{theorem}

\begin{proof}
As the first two implications are trivial ($(1) \Rightarrow (2)$ and $(2) \Rightarrow (3)$),
we prove $(3) \Rightarrow (1)$. 
Assume that $\pi \in K_{n}$ has no princes. We prove that for every $k\equiv 3$ $(mod$ $4)$ such that $k+1\leq n$: either $k,k-2,k+1,k-1$ or its reversal $k-1,k+1,k-2,k$ is a subsequence of $\pi$. 

We start with $k=3$:

Since $\pi$ has no princes, removing any element of $\pi$ will give us a $2$-block. 
That block cannot be a single strict $2$- block, otherwise by Lemma \ref{strict 2 block implies a prince}, $\pi$ has a prince. 

Hence, removing any element from $\pi$ will give us either a 
$3$-block or two strict $2$-blocks.  Note that the only way to get two 
strict $2$-blocks is by removing a separator of both types, though 
it is not true that whenever we remove such a separator we get two strict $2$-blocks.

We prove now that one of the consecutive subsequences $3142$ and $2413$ must appear in $\pi$. First, remove the element $1$ from $\pi$. Since $1$ cannot be a horizontal separator, its removal creates a $3$-block, and hence $\pi$ itself contains the consecutive subsequence $b,1,b+1,b-1$ or its reversal: $b-1,b+1,1,b$ for some $b \geq 3$. (The consecutive subsequence $b,1,b-1,b+1$ and its reversal are impossible since each element of $\pi$ is a separator, by the assumption that $\pi$ has no princes, 
but had the consecutive subsequence been $b,1,b-1,b+1$, the element $b-1$ clearly would not have been a vertical separator. Furthermore, if $b-1$ was a horizontal separator, $\pi$ would contain the consecutive subsequence: $b-2,b,1,b-1,b+1$ but then $b+1$ could not be a separator of any type). 
Now, we are done as soon as we show that $b=3$. In order to do that, consider the removal of $b+1$ from the subsequence $b,1,b+1,b-1$. If $b+1$ is a vertical separator then we must have $b=3$. Otherwise, $b+1$ is a horizontal separator and thus $\pi$ contains the consecutive subsequence $b+2,b,1,b+1,b-1$, but this cannot hold since 
in this case $b+2$ is not a separator of any type, which contradicts the assumption that $\pi$ has no princes. 

Now, for $k=7$, the same argument holds, provided that we start by removing $5$ instead of $1$. Since $4$ is confined between $1$ and $2$ in the case of $3142$ or between $4$ and $3$ in the case of $2413$, we have that $5$ can not be a horizontal separator, and we can continue as before. This procedure can be performed now sequentially for each $k$ such that $k=3(mod 4)$ and $k+1\leq n$. 

\end{proof}

We can now get an exact enumeration of the elements of $K_n$ that have no princes. 
\begin{cor}
The number of permutations in $K_n$ that have no princes is:

$$
\begin{cases}
2^k k! & {\text {if  } }   n=4k  \\
 0  &{\text otherwise}
\end{cases} \\$$

\end{cor}
\begin{proof}
According to Theorem \ref{3 equivalents}, if a permutation  $\pi \in K_n$  has no princes then there exist
$\pi^1,\dots, \pi^k \in \{[3142],[2413]\}$ and $\pi' \in S_k$ such that $\pi=\pi'[\pi^1,\dots,\pi^k]$. 
As a result, $n=4k$ and the permutation $\pi$ contains $k$ blocks of order 4, where for each such block we have two options: it can be of the form $[2413]$ or of the form $[3142]$. Thus there are $2^k k!$ options for such a permutation.
\end{proof}

Actually, the entire downset of an element that has no princes can be detected. So, let $\sigma \in K_n$ be such that $\sigma$ has no princes. By Theorem \ref{3 equivalents},   $\sigma=\sigma'[\sigma^1,\dots,\sigma^k]$, $\sigma^i\in \{[3142],[2143]\}$, $\sigma' \in S_k$ . We have the following:

\begin{cor}
Let $\sigma$ be as above. Then for each $\pi \in K_{l}$ ($l<n$) such that $\pi \prec \sigma$ we have $\pi=\pi'[\pi^1,\dots,\pi^m]$, $(m\leq k)$, where $\pi^i\in \{[3142],[2143],[1]\}$ and $\pi'\in S_m$ is such that $\pi' \prec \sigma'$. 
\end{cor}

\begin{proof}
In order to get $\pi \prec \sigma$ such that $\pi$ is a king, we must remove at least $3$
elements from a single block $\pi_i$.  
\end{proof}



The following Theorem adds some more information about the structure of the posets of king permutations. 

\begin{theorem}\label{chamsa}
Let $n>4$. For each $\sigma \in K_n$ there exists  $\pi \in K_5$ s.t. $\pi \preceq \sigma$. 
\end{theorem}

\begin{proof}
We prove by induction on $n$. The case $n=5$ being trivial. Assume that each king permutation of order $n-1$ contains a king permutation of order $5$ and prove this for king permutations of order $n$.
Let $\sigma \in K_n$. If $\sigma$ has a prince $\tau \in K_{n-1}$, then by the induction hypothesis there exists  $\pi \in K_5$ s.t. $\pi \preceq \tau \preceq \sigma$ and we are done. 
Otherwise, by Equation \ref{3 equivalents}, $\sigma$ is of the form $\sigma=\sigma'[\sigma^1,\dots,\sigma^k]$ with $\sigma^i \in \{[3142],[2413]\}$, $\sigma'\in S_k$. The first five elements of $\sigma$ are order isomorphic to
one of the permutations in the set 
$\{24135,31425,35241,42531\}\subseteq K_5$.

\end{proof}

\section{The M\"obius function of the poset of king permutations}
\label{mobious}

In this section we present some results regarding the M\"obius function of $\mathcal{K}$. We start with an example depicting the poset of the downset of the king permutation $[5246173]$. The red circled number above each permutation $\pi$ is the value of $\mu(\pi)$.

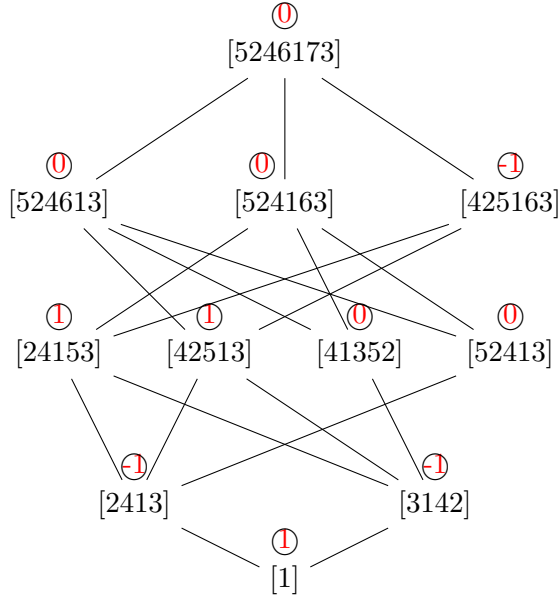
\begin{tikzpicture}
        \tikzstyle{every node} = [rectangle]
        \node (1) at (0,0) {$[1]$};
        \node at (0,0.5)  {\textcircled{\textcolor{red}1}};
        \node (2413) at (-2,1) {$[2413]$};
        \node at (-2,1.5)  {\textcircled{\textcolor{red}{-1}}};
        \node (3142) at (2,1) {$[3142]$};
        \node at (2,1.5)  {\textcircled{\textcolor{red}{-1}}};
        
        \node (24153) at (-3,3) {$[24153]$};
        \node at (-3,3.5)  {\textcircled{\textcolor{red}1}};
        
        \node (42513) at (-1,3) {$[42513]$};
        \node at (-1,3.5)  {\textcircled{\textcolor{red}1}};
        
        \node (41352) at (1,3) {$[41352]$};
        \node at (1.0,3.5)  {\textcircled{\textcolor{red}0}};
        
        \node (52413) at (3,3) {$[52413]$};
        \node at (3,3.5)  {\textcircled{\textcolor{red}0}};
        
        \node (425163) at (3,5) {$[425163]$};
        \node at (3,5.5)  {\textcircled{\textcolor{red}{-1}}};
        
        \node (524163) at (0,5) {$[524163]$};
        \node at (-0.3,5.5)  {\textcircled{\textcolor{red}0}};
        
        \node (524613) at (-3,5) {$[524613]$};
        \node at (-3,5.5)  {\textcircled{\textcolor{red}0}};
        
        \node (5246173) at (0,7) {$[5246173]$};
        \node at (0,7.5)  {\textcircled{\textcolor{red}0}};
        
        \foreach \from/\to in {1/3142, 1/2413,2413/24153,2413/42513,2413/52413,
        3142/24153,3142/42513,3142/41352,24153/425163,24153/524163,42513/425163,42513/524613,41352/524163,41352/524613,52413/524163,52413/524613,524613/5246173,524163/5246173,425163/5246173}
            \draw[-] (\from) -- (\to);
    \end{tikzpicture}

We start with the following cornerstone of our treatment of the M\"obius function on $\mathcal{K}$.
\begin{theorem}
\label{sroch}
Let $\pi \in K_n$, with $n >4$. If $[2413] \nprec \pi$ or $[3142]\nprec \pi$ then $\mu(\pi)=0$ in $\mathcal{K}$.  
\end{theorem}

\begin{proof}
We prove by induction on $n$. The basis is $n=5$ which can be easily checked. Assume that the claim holds for each 
$\pi \in K_l$, $4<l<n$. Let $\pi \in K_n$. By Theorem \ref{Each king contains 3142 or 2413}, we may assume without loss of generality that $\pi$ contains $[2413]$ but not $[3142]$.

Then $$\mu(\pi)=-\sum\limits_{\sigma \in [1,\pi)}{\mu(\sigma)}=-\sum\limits_{\sigma \in [1,[2413]]}{\mu(\sigma)}-\sum\limits_{\sigma \in ([2413],\pi)}{\mu(\sigma)}.$$ \\
The first summation is $0$ since the interval is closed, while each element in the second one is $0$ by the induction hypothesis. Hence we have $\mu(\pi)=0$. 
 
\end{proof}

For $\pi \in K_n$, define \[C(\pi)=\{\sigma \in \bigcup\limits_{l<n}{K_l}\mid \sigma \prec \pi,\text{ and there is no } \delta\in \bigcup\limits_{l<n}{K_l} \text{ such that } \sigma \prec \delta \prec \pi\}.\]

\begin{cor}
\label{cor42}
Let $\pi \in K_n$ be such that there is only one $\sigma \in C(\pi)$ such that $[3142] \prec \sigma$ and $[2413]\prec \sigma$. Then $\mu(\pi)=0$.
\end{cor}
\begin{proof}
Let $\sigma \in C(\pi)$ be the single element which contains both $[2413]$ and $[3142]$. 
$$\mu(\pi)=-\sum\limits_{\tau \in [1,\pi)}{\mu(\tau)}=-\sum\limits_{\tau \in [1,\sigma]}{\mu(\tau)}-\sum\limits_{\tau \in [1,\pi)-[1,\sigma]}{\mu(\tau)}.$$

The first summation vanishes since it runs through a closed interval, while the second summation vanishes by Theorem \ref{sroch}.

\end{proof}
For example:

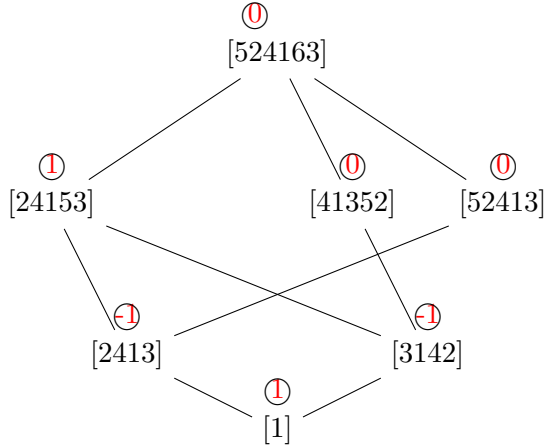
\begin{tikzpicture}
        \tikzstyle{every node} = [rectangle]
        \node (1) at (0,0) {$[1]$};
        \node at (0,0.5)  {\textcircled{\textcolor{red}1}};
        \node (2413) at (-2,1) {$[2413]$};
        \node at (-2,1.5)  {\textcircled{\textcolor{red}{-1}}};
        \node (3142) at (2,1) {$[3142]$};
        \node at (2,1.5)  {\textcircled{\textcolor{red}{-1}}};
        
        \node (24153) at (-3,3) {$[24153]$};
        \node at (-3,3.5)  {\textcircled{\textcolor{red}1}};

        \node (41352) at (1,3) {$[41352]$};
        \node at (1.0,3.5)  {\textcircled{\textcolor{red}0}};
        
        \node (52413) at (3,3) {$[52413]$};
        \node at (3,3.5)  {\textcircled{\textcolor{red}0}};

        \node (524163) at (0,5) {$[524163]$};
        \node at (-0.3,5.5)  {\textcircled{\textcolor{red}0}};

        \foreach \from/\to in {1/3142, 1/2413,2413/24153,2413/52413,
        3142/24153,3142/41352,24153/524163,41352/524163,52413/524163}
            \draw[-] (\from) -- (\to);
    \end{tikzpicture}

Corollary \ref{cor42} can be strengthen. In order to do that, we need the following: 

\begin{definition}

Let \[\mathbb{H}=\{[24153],[35142],[42513],[31524]\}.\] That is  $\mathbb{H}$ consists of all the elements of $K_5$ that contain both   $[2413]$ and $[3142]$.

\end{definition}

Note that in $K_5$, $\mu(\pi)=1$ if and only if $\pi \in \mathbb{H}$ (otherwise, by Theorem \ref{sroch}, $\mu(\pi)=0$).

\begin{theorem}
Let $\sigma \in K_n$ with $n>5$ such that there is only one $\pi \prec \sigma$ such that $\pi \in \mathbb{H}$ and for each $\pi' \prec \sigma$ such that $\pi \nprec \pi'$,  we have either $\pi'$ avoids $[3142]$, or $\pi'$ avoids $[2413]$. Then $\mu(\sigma)=0$.
\end{theorem}
\begin{proof}
We prove by induction on $n$. First, for $n=6$, let $\sigma \in K_6$ be such that the assumptions of the theorem are satisfied. Then, the permutation $\sigma$ has only one prince $\pi \in \mathbb{H}$. The other princes avoid $[2413]$ or avoid $[3142]$. Let $X=[1,\pi]$ and $Y=[1,\sigma)-X$. We have $$\mu(\sigma)=-\sum\limits_{\tau \in [1,\sigma)}\mu(\tau)=-\sum\limits_{\tau \in X}\mu(\tau)-\sum\limits_{\tau \in Y}\mu(\tau).$$ This sum vanishes since the first summation is over a closed interval while the second one vanishes by Theorem \ref{sroch}.

Now, let $n \geq 6$ and assume the validity of the claim for $5<k<n$. Let $\sigma \in K_n$ satisfy  our assumptions. Then
\begin{equation}
\label{form1}
 \mu(\sigma)=-\sum\limits_{\tau \in [1,\sigma)}\mu(\tau)=
-\sum\limits_{\textcolor{purple}{\tau \in (\pi, \sigma)}}\mu(\tau)
-\sum\limits_{\textcolor{red}{\tau \in[1,\pi]}}\mu(\tau)-
\sum\limits_{\textcolor{blue}{\tau \in Z}}\mu(\tau)   
\end{equation}
where $\textcolor{blue}{Z=[1,\sigma)-(\pi,\sigma)-[1,\pi]}$.
The elements of the first summation vanish by the induction hypothesis, the second one vanishes since it runs through a closed interval while the third one contains only elements which avoid $[2413]$ or avoid $[3142]$ and thus vanishes by Theorem \ref{sroch}. 
For example (the colors in the poset refer to the colors in Equation (\ref{form1}) :

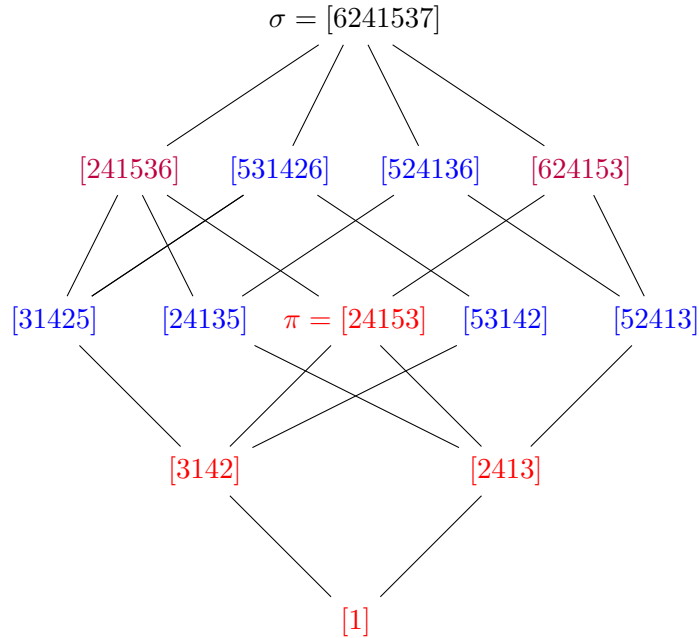
\begin{tikzpicture}
        \tikzstyle{every node} = [rectangle]
        
    \node (6241537) at (0,8) {$\sigma=[6241537]$};

        \node (241536) at (-3,6) {$\textcolor{purple}{[241536]}$};
        \node (531426) at (-1,6) {$\textcolor{blue}{[531426]}$};

        \node (524136) at (1,6) {$\textcolor{blue}{[524136]}$};
        \node (624153) at (3,6) {$\textcolor{purple}{[624153]}$};
        
        \node (31425) at (-4,4) {$\textcolor{blue}{[31425]}$};
        \node (24135) at (-2,4) {$\textcolor{blue}{[24135]}$};
        \node (24153) at (0,4) {$\textcolor{red}{\pi=[24153]}$};
        \node (53142) at (2,4) {$\textcolor{blue}{[53142]}$};
        \node (52413) at (4,4) {$\textcolor{blue}{[52413]}$};
        
        \node (3142) at (-2,2) {$\textcolor{red}{[3142]}$};
        \node (2413) at (2,2) {$\textcolor{red}{[2413]}$};
        \node (1) at (0,0) {$\textcolor{red}{[1]}$};

        \foreach \from/\to in {1/3142,1/2413,2413/24135,2413/24153,2413/52413,3142/31425,3142/24153,3142/53142,31425/241536,31425/531426,24135/241536,24135/524136,24153/241536,24153/624153,53142/531426,31425/531426,52413/524136,52413/624153,241536/6241537,531426/6241537,524136/6241537,624153/6241537}
        \draw[-] (\from) -- (\to);
    \end{tikzpicture}
    \end{proof}

\section{Further directions}

\begin{enumerate}
\item 
We saw that a $\pi \in S_n$ being a king permutation is equivalent to having $br(\pi)\geq 3$. 
It might be interesting to characterize from the point view of permutation patterns the permutations  $\pi \in S_n$ having $br(\pi)\geq d$ for $d>3$. In \cite{BHW}, Blackburn, Homberger and Winkler showed that when $d$ is fixed and $n \rightarrow \infty$, the probability that $br(\pi) \geq d+2$ tends to $e^{-d^2-d}$. 
A challenge may be to explore the structure of the sub-poset consisting of these permutations for constant $d$. 

\item 
 A natural question is to extend the Hertzsprung’s problem of finding the number of ways to arrange $n$ non-attacking kings on an $n \times n$ chess board to a variant of the chess game, namely the cylindrical chess,  in which the right and left edges of the board are imagined to be joined in such a way that a piece can move off one edge and reappear at the opposite edge, moving in the same or a parallel line. Some steps in this direction were made by the authors of this article in \cite{counting}.

\end{enumerate}



\begin{thebibliography}{}


\bibitem{AD}{P. Adeline, R. Dominique},{ Simple permutations poset},  https://arxiv.org/pdf/1201.3119.pdf, 2012.

\bibitem{Adin} {R. Adin, E. Bagno, E. Eisenberg, S. Reches, M. Sigron} ,{ On Two-Sided Gamma-Positivity for Simple Permutations}, The Electronic Journal of Combinatorics. Volume 25 Issue 2, 2018.

\bibitem{BHW}, {S.R. Blackburn, C. Homberger, P. Winkler}, {The minimum Manhattan distance and minimum jump of permutations},
Journal of Combinatorial Theory, Series A, 161, 364-386.

\bibitem{BBL}{P. Bose, J. F. Buss, A. Lubiw}, { Pattern Mathings for permutations}, Inf. Process. Lett. 65(5): 277-283 (1998).
\bibitem{BERS}{E. Bagno, E. Eisenberg, S. Reches, M. Sigron},{Separators - a new statistic for permutations}, https://arxiv.org/pdf/1905.12364.pdf, 2019.

\bibitem{counting}{Counting King Permutations on the Cylinder},
{E. Bagno, E. Eisenberg, S. Reches and M. Sigron
}, arXiv:2001.02948v1


\bibitem{BJKM}{R. Brignall, V.J.Jel\'{i}nek, J. Kyn\^{c}l, and D. Marchant}, {Zeroes of the M\"{o}bius function of pemutations} Mathematika, 65(4),(2019) 1074-1092. doi:10.1112/S0025579319000251


\bibitem{BHE}{D. Bevan, C. Homberger and B. E. Tenner},{ Prolific permutations and permuted packings: downsets containing many large patterns},Journal of Combinatorial Theory, Series A, Volume 153, January 2018, Pages 98-121.



\bibitem{BM} {R. Brignall and D. Marchant}, {The M\"{o}bius function of permutations with an indecomposable lower bound}. Discrete Math. 341(5) 2018, 1380–1391.10.1016/j.disc.2018.02.012.


\bibitem{H}{C. Homberger},{ Patterns in Permutations and Involutions, A Structural and Enumerative Approach},University of Florida, ProQuest Dissertations Publishing, 2014.3647831

\bibitem{Ro}{D. P. Robins},{ The probability that neighbors remain neighbors after random rearrangements},  The American Mathematical Montly, Vol. 87, 1980 - issue 2. pp 122-124. 


\bibitem{Sl}{Sloane, N. J. A.}, Sequence A002464 in "The On-Line Encyclopedia of Integer Sequences."




\bibitem{Vatter}{V.\ Vatter},{ Permutation classes},
in: Handbook of Combinatorial Enumeration (M.\ Bona,
ed.), CRC Press 2015, pp.\ 753--833.


\end{thebibliography}
\end{document}